\documentclass[reqno,centertags,12pt]{amsart}
    \usepackage{amsmath,amsthm,amscd,amssymb,latexsym,verbatim}
    \usepackage{amssymb}

    \usepackage{graphicx,epsf,cite}
    
    \usepackage[shortlabels]{enumitem}
    
    \usepackage{xcolor}
    
    -\marginparwidth \oddsidemargin -4mm \evensidemargin \oddsidemargin

    \newtheorem{theorem}{Theorem}[section]

    \newtheorem{lemma}[theorem]{Lemma}
    
    \theoremstyle{definition}
    \newtheorem{definition}[theorem]{Definition}

    \newcounter{smalllist}

    \allowdisplaybreaks
    \numberwithin{equation}{section}
    
    \newcommand{\lb}{\label}

    \newcommand{\supp}{\text{\rm{supp}}}

    \newcommand{\beq}{\begin{equation}}
    \newcommand{\eeq}{\end{equation}}
    
    \newcommand{\bal}{\begin{align}}
    \newcommand{\eal}{\end{align}}
    \newcommand{\bals}{\begin{align*}}
    \newcommand{\eals}{\end{align*}}
    
    \newcommand{\bbN}{{\mathbb{N}}}
    \newcommand{\bbR}{{\mathbb{R}}}
    
    \newcommand{\bbP}{{\mathbb{P}}}
    \newcommand{\bbE}{{\mathbb{E}}}
    \newcommand{\bbZ}{{\mathbb{Z}}}

    \newcommand{\bbS}{{\mathbb{S}}}

    \newcommand{\calS}{{\mathcal S}}

    \newcommand{\calB}{{\mathcal B}}
    
    \newcommand{\calL}{{\mathcal L}}
    
    \newcommand{\calF}{{\mathcal F}}

    \newcommand{\eps}{\varepsilon}

\begin{document}
    \title[Homogenization for Space-Time-Dependent KPP Reactions]
    {Homogenization for Space-Time-Dependent KPP Reaction-Diffusion Equations and G-Equations}
    
    \author{Yuming Paul Zhang and Andrej Zlato\v s}
    
    \address{\noindent Department of Mathematics \\ University of
    California San Diego \\ La Jolla, CA 92093 \newline Email: \tt
    yzhangpaul@ucsd.edu, zlatos@ucsd.edu}
    
    
    
    \begin{abstract} 
We prove stochastic homogenization for reaction-advection-diffusion equations with random space-time-dependent KPP reactions with  temporal correlations that are decaying in an appropriate sense.  
We show that the limiting homogenized dynamic has the simple form of spreading with some deterministic direction-dependent speeds from the support of the initial datum.  We obtain analogous results for  G-equations with random flame speeds and incompressible background advections.
Important ingredients in our proofs are a non-autonomous subadditive theorem and the principle of virtual linearity for  KPP reactions from the companion papers \cite{ZhaZla4, ZlaKPPspread}.
    \end{abstract}
    
    \maketitle

    \section{Introduction and Main Results} \lb{S1}

We study long time behavior of solutions to models of reactive processes, such as combustion and population dynamics, in random environments --- specifically, reaction-diffusion equations and G-equations.  The former are the PDE
\beq\lb{1.0}
u_t=\calL_\omega u+f(t,x,u,\omega),
\eeq
with $f:\bbR^{d+1}\times[0,1]\times\Omega\to\bbR$ some non-linear reaction function, a second-order linear term
\beq\lb{1.0'}
\calL_\omega u(t,x):=\sum_{i,j=1}^d A_{ij}(t,x,\omega)u_{x_ix_j}(t,x)+\sum_{i=1}^d b_i(t,x,\omega)u_{x_i}(t,x),
\eeq
and $\omega$ an element from some probability space $(\Omega,\bbP,\calF)$.
One also typically assumes that $f$ vanishes at $u=0,1$, and  solutions $0\le u\le 1$ represent normalized temperature or density, which is subject to reaction, advection, and diffusion.  The simplest reaction-diffusion model involves $\calL_\omega \equiv \Delta_x$ and $f=f(u)$, but we will consider here the general non-isotropic, space-time-dependent,  random reaction-advection-diffusion setting  of \eqref{1.0}.

We will mainly concentrate on this case, but our methods equally apply to the related first-order flame propagation model 
\beq\lb{G1.1}
u_t+v(t,x,\omega)\cdot \nabla u=c(t,x,\omega)|\nabla u|.
\eeq
This Hamilton-Jacobi PDE is called  the G-equation (it is often considered with $c\equiv 1$ only), where $c>0$ is the flame speed  and   $v$ is some (incompressible) background advection.

We will consider \eqref{1.0} with the KPP (a.k.a.~Fisher-KPP) reactions,  first studied by Kolmogorov, Petrovskii, and Piskunov \cite{KPP} and Fisher \cite{Fisher} in 1937.  We will therefore assume the following uniform KPP hypotheses.

\begin{definition} \lb{D.1.0}
A Lipschitz function $f:\bbR^{d+1}\times [0,1] \times\Omega\to\bbR$ is a {\it KPP reaction} if $f(\cdot,\cdot,0,\cdot)\equiv 0\equiv f(\cdot,\cdot,1,\cdot)$ and  $f(t,x,u,\omega)\le f_u(t,x,0,\omega)u $ for all $(t,x,u,\omega)\in \bbR^d\times [0,1]  \times \Omega$ (with $f_u(\cdot,\cdot,0,\cdot)$ existing pointwise), plus the following uniform hypotheses hold. 
We have  $\inf_{(t,x,\omega)\in\bbR^{d+1}\times\Omega}f(t,x,u,\omega)>0$ for each $u\in(0,1)$, as well as $\inf_{(t,x,\omega)\in\bbR^{d+1}\times\Omega}f_u(t,x,0,\omega)>0$ and   
\beq\lb{0.5}
\lim_{u\to 0}\sup_{(t,x,\omega)\in\bbR^{d+1}} \left( f_u(t,x,0,\omega)-\frac{f(t,x,u,\omega)}{u} \right)=0.
\eeq
\end{definition}

When physical processes occur in random media, one often expects them to exhibit an effectively homogeneous dynamic on large space-time scales due to large-scale averaging of the variations in the environment.  Our main results show that this phenomenon, called {\it homogenization}, indeed occurs for \eqref{1.0} and \eqref{G1.1} in very general settings under suitable hypotheses.  The main two of the latter are always stationarity of the environment and some mixing assumption on it, without which one cannot reasonably hope for homogenization to occur.  We state our versions of these next, with $H$ being either $(A,b,f_u(\cdot,\cdot,0,\cdot))$ or $(c,v)$ (see below for why it suffices to only include $f_u(\cdot,\cdot,0,\cdot)$  here in the KPP reaction case).

\begin{definition}\lb{D.2.3}
Let  $(\Omega,\calF,\bbP)$ be some probability space and $H$ a measurable function on $\bbR^{d+1}\times\Omega$ with values in some measurable space.  We say that $H$ is {\it space-time stationary} if there is a group of measure-preserving bijections $\{{\Upsilon_{(s,y)}:\Omega\to\Omega}\}_{(s,y)\in \bbR^{d+1}}$ with $\Upsilon_{(0,0)}={\rm Id}_{\Omega}$ and  $\Upsilon_{(s,y)}\circ\Upsilon_{(r,z)}=\Upsilon_{(s+r,y+z)}$ for any $(s,y),(r,z)\in \bbR^{d+1}$, and  for any $(t,x,s,y,\omega)\in   \bbR^{2d+2}\times\Omega$ we have
\beq \lb{0.3}
H \left(t,x,\Upsilon_{(s,y)}\omega \right)=H(t+s,x+y,\omega).
\eeq
For any $t\in\bbR$, we  let $\calF_t^\pm(H)$ 
be the $\sigma$-algebra generated by the family of random variables 
\[
\left\{ H(s,x,\cdot)\,\big|\, \pm(s-t)\geq 0 \text{ and } x\in\bbR^d\right\}.
\]
We also define for each $s\ge 0$,
\[
\phi_H(s):=\sup\left\{\left| \bbP[F|E] -\bbP[F]\right| \,\big|\,\, t\in\bbR \,\,\&\,\, (E,F)\in \calF_t^-(H) \times \calF_{t+s}^+(H) \,\,\&\,\, \bbP[E]>0 \right\}.
\]
\end{definition}

\smallskip
So $\phi_H$ is clearly non-increasing, and it vanishes at some $s\ge 0$ precisely when $\calF^-_t(H)$ and $\calF^+_{t+s}(H)$  are $\bbP$-independent for each $t\in\bbR$ (in that case $H$ has a {\it finite temporal range of dependence}).
The mixing hypothesis mentioned above will in our case be the assumption that $\lim_{s\to\infty} \phi_H(s)=0$ (possibly at some rate)  for the appropriate function $H$.
That is, it will involve mixing in time but not necessarily in space.

Long-time propagation of solutions to  \eqref{1.0} is well known to be ballistic, with solutions converging locally uniformly to 1.  One should therefore expect homogenization to take the following form.   First, solutions starting from compactly supported initial data should approximate the characteristic function of $t\calS$ for some open bounded convex set $\calS\ni 0$ (called {\it Wulff shape}) as $t\to\infty$.  This means that, as $t\to\infty$ the $\theta$-level set of the solution should, after scaling by $\frac 1t$ in space, converge in Hausdorff distance to $\partial \calS$ for each $\theta\in(0,1)$.  And, of course, this should hold for a large set of $\omega\in\Omega$ in the probabilistic sense, with the Wulff shape being deterministic (i.e., $\omega$-independent).

Second, \eqref{1.1} should exhibit  a homogenized large-scale dynamic in the ballistic scaling
\begin{equation}\label{0.5'}
u^{\eps}(t,x, \omega):=u\left( \eps^{-1}  t, \eps^{-1}  x, \omega\right),
\end{equation}
with $\eps>0$ small.  This of course turns \eqref{1.0} into its large-space-time-scale version
\begin{equation}\label{3.0}
u^{\eps}_{t}= \calL_\omega^\eps u^{\eps}+ \eps^{-1} f\left(\eps^{-1} t,\eps^{-1} x, u^{\eps}, \omega\right),
\end{equation}
where
\[
\calL_\omega^\eps u^\eps(t,x):= \eps \sum_{i,j=1}^d A_{ij} \left(\eps^{-1} t,\eps^{-1} x,\omega \right) u^\eps_{x_ix_j}(t,x)  + \sum_{i=1}^d b_i \left(\eps^{-1} t,\eps^{-1} x,\omega \right) u^\eps_{x_i} (t,x).
\]
Then one hopes that,  again for a large set of $\omega\in\Omega$, solutions to \eqref{3.0} with some $\eps$-independent initial datum $u_0$ converge as $\eps\to 0$  to a function $\bar u$ that solves some homogeneous PDE with the same initial value (the term ``homogenization'' usually refers to this type of result).  

Such stochastic homogenization results were obtained previously in several works for time-independent $(A,b,f)$  in one spatial dimension, where the geometry of the  level sets of solutions is trivial (they are typically two points ballistically  traveling to $\pm\infty$).  The interested reader can consult, for instance,  papers \cite{BarSou, BerNad, GF, MajSou, NolRyz, VakVol, ZlaGenfronts, ZlaBist} and references therein  (yet others involve spatially periodic rather than random $(A,b,f)$), which study KPP reactions as well as ignition and bistable reactions (for which $f (\cdot,\cdot,u,\cdot)$ vanishes or is negative when $u>0$ is close to 0).  

Progress in the multi-dimensional (and still time-independent) case $d\ge 2$ has been much more limited, due to the geometry of the level sets of solutions substantially complicating the analysis.   Stochastic homogenization results for stationary ergodic ignition reactions and $(A,b)=(\Delta,0)$  in dimensions $d\le 3$ were recently obtained by the second author and Lin  \cite{LinZla} as well as by both authors  \cite{ ZhaZla, ZhaZla3} (homogenization results in spatially periodic multidimensional media appear in, e.g.,  \cite{AlfGil,BarSou, CLM, GF, LinZla, MajSou}), but the only such results for KPP reactions that we are aware of are Theorem 9.3 in \cite{LioSou} by Lions and Souganidis, and 
Theorem 1.4 in the companion paper \cite{ZlaKPPspread} by the second author (the latter even holds in the time-periodic $(A,b,f_u(\cdot,\cdot,0,\cdot))$ case, which is closely related to the time-independent setting).  However, we note that Theorem 9.3 in \cite{LioSou}  is stated without a proof, and the authors  only indicated that methods developed by them and in other works can be used to obtain one.  Moreover, we know of no other prior homogenization results  even in the simpler case of  time-independent and spatially periodic KPP reactions (although existence of Wulff shapes and front speeds in the periodic case goes back to work of G\" artner and Freidlin \cite{GF}).

In light of the  above discussion, our main result for KPP reactions (Theorem \ref{T.1.0} below) appears to be the first one in the general time-dependent setting for any reaction and in any dimension.  
In fact, homogenization results in time-dependent environments seem to be rather sparse even in the much more studied and developed setting of  Hamilton-Jacobi equations (see below).
We note that the proof of Theorem \ref{T.1.0} uses two new ingredients, a  {\it non-autonomous} version of the classical Kingman's subadditive ergodic  theorem \cite{kingman}  (Theorem \ref{T.1.1} below) and the principle of {\it virtual linearity} for  \eqref{1.0}, from the companion papers \cite{ZhaZla4, ZlaKPPspread}.

In addition, together with Theorem 1.4 in  \cite{ZlaKPPspread}, Theorem \ref{T.1.0} appears to be the first multi-dimensional stochastic homogenization result for \eqref{1.0} that provides an explicit formula for the solution to the homogenized dynamic (except in the special case of isotropic ignition reactions, see below).  The results in \cite{LinZla, ZhaZla, ZhaZla3} for ignition reactions, as well as Theorem 9.3 in \cite{LioSou} for KPP reactions show that in the relevant settings, solutions to \eqref{3.0} with common initial datum $ u_0$ converge as $\eps\to 0$ to a discontinuous viscosity solution to the homogeneous  Hamilton-Jacobi equation
\beq\lb{0.7}
\bar u_t=c^* \left(-\nabla\bar u |\nabla \bar u|^{-1} \right) |\nabla \bar u|
\eeq
that only takes values in $\{0,1\}$ for all $t>0$, where $c^*(e)$ is some $(A,b,f)$-dependent deterministic {\it front speed} in direction $e\in\bbS^{d-1}$ (see \cite{LinZla,ZhaZla, ZlaKPPspread} for its definition).  This yields an implicit formula for  the homogenized solutions.  However, here and in \cite{ZlaKPPspread} we show that for  KPP reactions (including in the  time-dependent case), one in fact has the explicit formula
\beq\lb{0.6}
\bar u := \chi_{G+t\calS},
\eeq
where (essentially) $G:=\supp \, u_0$ and $\calS$ is the Wulff shape for $(A,b,f)$ (this then also implies that $c^*(e)$ exists for each $e\in\bbS^{d-1}$ and  $c^*(e)=\sup_{y\in \calS} \,y\cdot e$).  Moreover, we show that the dependence of $\calS$ on $f$ is only through $f_u(\cdot,\cdot,0,\cdot)$ in the KPP reaction case.

The reason for $\bar u$ only taking values in $\{0,1\}$ is the hair-trigger effect, discussed in Section~\ref{s.3} below, which shows that solutions to \eqref{1.0} transition from values arbitrarily close to 0 to those arbitrarily close to 1 in $\eps$-independent time.  This then becomes an instantaneous transition from value 0 to 1 in the $\eps\to 0$ limit for \eqref{3.0}.  However, our proofs  show that this transition also becomes sharp in space in this limit, which is not surprising but also not an obvious corollary.  Of course, this means that  for solutions to \eqref{1.0}, spatial transition from values close to 0 to those close to 1 happens on distances of size $o(t)$.  This shows that in the setting of \eqref{1.0}, it makes most sense to consider initial data that are also characteristic functions of sets in $\bbR^d$, but our main results in fact hold for more general initial data (see \eqref{3.0'} below).

This contrasts with the case of ignition reactions in dimensions $d\le 3$, where the second author proved that the above spatial transition for \eqref{1.0} occurs on distances of size $O(1)$ \cite{ZlaInhomog} (calling this the {\it bounded width} property of solutions).  For this it is crucial that the hair trigger effect is not present for ignition reactions, and the argument was based on the solution dynamic being {\it pushed} for ignition reactions when $d\le 3$ (which may fail when $d\ge 4$ \cite{ZlaInhomog}).  On the other hand, for KPP reactions it is  {\it pulled} due to the crucial hypothesis $f(t,x,u,\omega)\le f_u(t,x,0,\omega)u $, which guarantees that the dynamic depends on $f$ only through $f_u(\cdot,\cdot,0,\cdot)$ \cite{ZlaKPPspread}.  See \cite{ZlaInhomog, ZlaKPPspread} for details on these concepts and further discussion.

We note that the explicit formula \eqref{0.6} also holds for time-independent stationary ergodic reactions if \eqref{1.0} has a Wulff shape $\calS$ and this $\calS$ has no corners \cite{LinZla} (i.e., it has a unique unit outer normal vector at each $x\in\partial\calS$).  However, the latter hypothesis has previously only been verified for isotropic ignition reactions in dimensions $d\le 3$ \cite{LinZla}, when $\calS$ is a ball (this clearly also holds in the settings of Theorem \ref{T.1.0} below and Theorem 1.4 in  \cite{ZlaKPPspread}), and it is known that it can fail even for non-isotropic periodic ignition reactions in two dimensions.  In fact, an example constructed by Caffarelli, Lee, and Mellet in  \cite{CLM} was  used in \cite{LinZla} to show that not only $\calS$ can have corners in this setting, but \eqref{0.6} can also fail for non-KPP reactions.

Let us now state our main result for \eqref{1.0}, in which we can also accommodate $\eps$-dependent shifts $y_\eps$ of the initial value and perturbations that decay in an appropriate sense as $\eps\to 0$.  To simplify the relevant notation, let $B_r:=B_r(0)\subseteq\bbR^d$ for $r>0$ and $B_0:=\{0\}$, then let $B_r(G):=G+B_r$ and $G^0_r:=G\backslash\overline{B_r(\partial G)}$ for $G\subseteq\bbR^d$ and $r\ge 0$ (so $G^0_0$ is the interior of $G$).

\begin{theorem} \lb{T.1.0}
Let $f$ be a KPP reaction and let $\calL_\omega$ be from \eqref{1.0'}, where
 $A=(A_{ij})$ a bounded symmetric matrix with $A\ge \lambda I$ for some $\lambda>0$, and the vector $b=(b_1,\dots,b_d)$ satisfies 
 \beq \lb{0.2}
 \|b\|_{L^\infty}^2< 4 \lambda \inf_{(t,x,\omega)\in\bbR^{d+1} \times\Omega} f_u(t,x,0,\omega).
 \eeq
  Also assume that $H:=(A,b,f_u(\cdot,\cdot,0,\cdot))$ is space-time stationary.
 
 (i)  If   $\lim_{s\to\infty}s^\alpha \phi_H(s)=0$ for some $\alpha>0$, then there is a convex bounded open set $\calS\subseteq\bbR^d$ containing 0 (called Wulff shape), which depends only on $H$, and  the following holds for almost all $\omega\in\Omega$. If $G\subseteq \bbR^d$ is open, $\theta\in (0,1)$, $\Lambda<\infty$, and $u^\eps(\cdot,\cdot,\omega)$ solves \eqref{3.0}
with
\beq\lb{3.0'}
\theta\chi_{(G+y_\eps)^0_{\rho(\eps)}}\le u^\eps(0,\cdot,\omega)\leq \chi_{B_{\rho(\eps)}(G+y_\eps)}
\eeq
for each $\eps>0$, with some $y_\eps\in B_\Lambda$ and $\lim_{\eps\to0}\rho(\eps)=0$  (when $y_\eps=0$ and $\rho(\eps)= 0$, this becomes just $\theta\chi_{G} \le u^\eps(0,\cdot,\omega)\le \chi_{G}$), then
\beq \lb{0.1}
\lim_{\eps\to0}u^\eps(t,x+y_\eps,\omega)=\chi_{G^\calS}(t,x)
\eeq
locally uniformly on $([0,\infty)\times\bbR^d)\backslash \partial G^\calS$, where $G^\calS:=\{(t,x)\in\bbR^+\times\bbR^d\,|\,x\in G+t\calS\}$.

(ii)  If $\lim_{s\to\infty} \phi_H(s)=0$, then (i) holds with $\Lambda=\infty$ and with \eqref{0.1} replaced by 
\beq\lb{3.6}
\lim_{\eps\to0} \bbP\left[ (G+(1-\delta)t \calS)\cap B_{\delta^{-1}}\subseteq \Gamma_{\theta'}^{\eps}(t,\cdot)\cap B_{\delta^{-1}}\subseteq (G+(1+\delta)t \calS)\cap B_{\delta^{-1}}\,\, \forall t\in [\delta,\delta^{-1}]\right]=1
\eeq
for any $\delta,\theta'\in(0,1)$, where $\Gamma_{\theta'}^{\eps}(t,\omega):=\{x\in\bbR^d\,|\, u^\eps(t,x+y_\eps,\omega)\geq \theta' \}$.
\end{theorem}

{\it Remarks.} 
1.  The hypothesis on $\phi_H$ is of course satisfied in both (i) and (ii) when $H$ has a finite temporal range of dependence.
\smallskip

2.  Since the homogenized dynamic only depends on $f$ via $f_u(\cdot,\cdot,0,\cdot)$, the full reaction $f$ need not be space-time-stationary or have the required temporal dependence properties.
\smallskip


3. It is shown in \cite{ZlaKPPspread} that the bound \eqref{0.2} is necessary (and sharp) for solutions to spread with positive  speeds in all directions (i.e., for $0\in\calS$).
\smallskip

4.  Allowing for $y_\eps\neq 0$ makes (i) more general, but this is not the case for (ii) due to space stationarity of $H$.
\smallskip

5.  In the course of the proof we also show in Theorems \ref{T.2.5} and \ref{T.2.6} below that $\calS$ is the Wulff shape for \eqref{1.1} in the sense of propagation  from compactly supported initial data.
\smallskip

6.  In (ii) we also have that $\limsup_{\eps\to0}u^\eps(t,x+y_\eps,\omega)\le \chi_{G^\calS}(t,x)$ locally uniformly on $([0,\infty)\times\bbR^d)\backslash \partial G^\calS$ (see the remark at the end of Section \ref{s.3}).
\smallskip


%

%
%

\medskip

Let us now turn to the  G-equation \eqref{G1.1}. The scaling \eqref{0.5'} transforms it into
\beq\lb{4.0}
u^\eps_t+v(\eps^{-1}t,\eps^{-1}x,\omega)\cdot \nabla u^\eps=c(\eps^{-1}t,\eps^{-1}x,\omega)|\nabla u^\eps|,
\eeq
and the goal is again to show that the dynamic for this PDE converges in an appropriate sense to that for some deterministic homogeneous equation  as $\eps\to 0$. 

The G-equation is a (first-order) Hamilton-Jacobi equation and there is a vast literature on periodic and stochastic homogenization for general Hamilton-Jacobi equations 
\[
u_t=H(t,x,\nabla u,\omega),
\]
as well as their second-order (viscous) analogs.
We will not attempt to review it here, and will only focus  on homogenization results involving time-dependent Hamiltonians. 
Kosygina and Varadhan proved homogenization for space-time stationary ergodic super-linear (in $p$) Hamiltonians in the presence of diffusion represented by the Laplacian  \cite{KosVar}, Schwab  addressed the same case but without diffusion \cite{Sch}, and
Jing, Souganidis, and Tran   treated the cases of space-time stationary ergodic super-quadratic Hamiltonians with possibly degenerate diffusions \cite{HJ_JST}.  The last three authors also considered \eqref{G1.1} with $v\equiv 0$ and $c$ that is either periodic in time and stationary-ergodic in space or vice versa \cite{JST2,jing}.
All these papers considered Hamiltonians that are convex and coercive in $\nabla u$  (the latter means that $\lim_{p\to \infty}H(t,x,p,\omega)=\infty$ for all $(t,x,\omega)\in\bbR^{d+1}\times\Omega$), which is a frequent hypothesis in the theory, even in the time-independent case.

The Hamiltonian $H(t,x,p,\omega):=c(t,x,\omega)|p|-v(t,x,\omega)\cdot p$ from \eqref{G1.1} is convex when $c\ge 0$, but it is only coercive when $|v|<c$.
Hence none of the above results are applicable to the G-equation when this fails, and we in fact know of only two prior homogenization results in the non-coercive time-dependent case. Cardaliaguet, Nolen, and Souganidis obtained homogenization for \eqref{G1.1} with $c\equiv 1$ and space-time periodic $v$ with a not-too-large divergence  \cite{CNS}  (independently, Xin and Yu addressed the divergence-free time-independent space-periodic case at the same time \cite{XinYu}; see also the work of Cardaliaguet and Souganidis  \cite{CarSou} for the general spatially stationary ergodic case).  More recently, Burago, Ivanov, and Novikov proved it with $c\equiv 1$ and space-time stationary divergence-free $v$ that is not too large in average over very large balls (specifically, \eqref{Gc4} below holds with $c\equiv 1$) and has a finite temporal range of dependence \cite{BIN}.  Hence this was the first (and prior to our results the only) stochastic homogenization result in the non-coercive time-dependent setting.


Our approach to homogenization for  KPP reaction-advection-diffusion equations, via the non-autonomous subadditive theorem from the next section, turns out to easily extend to the setting of $G$-equations with general $(c,v)$ that have infinite temporal ranges of dependence, provided their temporal correlations decay in an appropriate sense.
%
Our main result for \eqref{G1.1} is Theorem~\ref{T.4.0} below, which we discuss next.
We note that besides the ability to accommodate some environments with infinite temporal ranges of dependence, another advantage of our method is that it applies to second-order equations, including \eqref{1.0} (the method in \cite{BIN} does not seem to be fully extendable to this setting, due to the need for a control representation formula for solutions, such as \eqref{5.5} below).  It will be used to study homogenization for other (viscous) Hamilton-Jacobi PDE elsewhere \cite{ZhaZla6}.

In the setting of G-equations, we again have the concept of a Wulff shape, which is now the asymptotic shape of  the reachable sets in the sense of the following definition.

\begin{definition}
We say that $(t_1,x_1)\in\bbR^{d+1}$ is  $\omega${\it -reachable}  from $(t_0,x_0)\in(-\infty,t_1]\times\bbR^{d}$ if there is an absolutely continuous path $\gamma:[t_0,t_1]\to\bbR^d$ such that $\gamma(t_j)=x_j$ ($j=0,1$) and
\[
\left|\gamma'(t)-v(t,\gamma(t),\omega)\right|\leq c(t,\gamma(t),\omega)
\]
for almost all $t\in [t_0,t_1]$. For any $t\geq 0$, we let
\beq\lb{G1.3}
\Gamma(t,\omega;t_0,x_0):=\left\{x\in\bbR^d\,\big|\, (t_0+t,x)\text{ is $\omega$-reachable from }(t_0,x_0)\right\}
\eeq
be the $\omega${\it -reachable set} from $(t_0,x_0)$ at time $t$, and denote $\Gamma(t,\omega):=\Gamma(t,\omega;0,0)$.
\end{definition}

These sets allow one to explicitly solve \eqref{G1.1} via a well known control representation formula  (see, e.g., Theorem 7.2 in \cite{control}), under reasonable hypotheses on $c$ and $v$, so that after scaling we obtain
\beq\lb{5.5}
u^\eps(t,x,\omega)=\sup_{x\in \eps\Gamma(\eps^{-1}t,\omega;0,\eps^{-1}y)}  u^\eps(0,y,\omega) 
\eeq
for solutions to \eqref{4.0}.
If there is $\calS\subseteq\bbR^d$ such that $\Gamma(t,\omega)$ approaches $t\calS$ as $t\to\infty$, for a large set of $\omega$ in the probabilistic sense (space-time stationarity then shows that the same holds for $\Gamma(t,\omega;t_0,x_0)$ for any $(t_0,x_0)\in\bbR^{d+1}$), then $\calS$ is the Wulff shape for \eqref{G1.1}.  In this case it follows from \eqref{5.5} that if the solutions $u^\eps$ share the same initial datum $u_0$, then they converge in an appropriate sense to the function
\beq\lb{5.6}
\bar{u}(t,x):=\sup_{x\in y+t\calS} u_0(y) = \sup_{y\in x-t\calS} u_0(y),
\eeq
which again also solves \eqref{0.7} with $c^*(e)=\sup_{y\in \calS} \,y\cdot e$.



We note that unlike for \eqref{1.0}, here the transition time from one value of $u$ to another should be roughly proportional to the inverse of the spatial gradient of the solution, and therefore can increase as $O(\frac 1\eps)$ if the solution gradient is $O(\eps)$.  It then makes perfect sense to consider initial data for \eqref{4.0} equal to or approximating some continuous function as $\eps\to 0$, which is what we will therefore do  in the following analog of Theorem \ref{T.1.0} for the G-equation (nevertheless, our arguments easily extend to discontinuous initial data).

%

\begin{theorem} \lb{T.4.0}
 Let $(c,v)$ be bounded, uniformly continuous in $t$, and (uniformly) Lipschitz in $x$,
with $v$ divergence-free (i.e., $\nabla_x \cdot v(t,\cdot,\omega)=0$ holds a.e., for all $(t,\omega)\in\bbR\times\Omega$) and
\beq\lb{Gc4}
\inf_{L>0}\sup_{(t,x,\omega)\in\bbR^{d+1}\times\Omega}\left\|\frac{1}{L^d}\int_{[0,L]^d}v(t,x+y,\omega)dy\right\|_{L^\infty} 
<\inf_{(t,x,\omega)\in\bbR^{d+1}\times\Omega}c(t,x,\omega).
\eeq
Also assume that $H:=(c,v)$ is space-time stationary.

 (i)  If   $\lim_{s\to\infty}s^\alpha \phi_H(s)=0$ for some $\alpha>0$, then there is a convex bounded open set $\calS\subseteq\bbR^d$ containing 0 (called Wulff shape)
such that  the following holds for almost all $\omega\in\Omega$. 
If $u_0$ and $u^\eps(0,\cdot,\omega)$ for each $(\eps,\omega)\in (0,1)\times\Omega$ are uniformly continuous on $\bbR^d$,  $\Lambda<\infty$, and $u^\eps(\cdot,\cdot,\omega)$ solves \eqref{4.0} in the viscosity sense with
\beq\lb{5.1'}
\sup_{\omega\in\Omega} \|u^\eps(0,\cdot+y_\eps,\omega)- u_0\|_{L^\infty} \leq \rho(\eps)
\eeq
for some $y_\eps\in B_\Lambda$ and $\lim_{\eps\to0}\rho(\eps)=0$  (when $y_\eps=0$ and $\rho(\eps)= 0$, this becomes just $u^\eps(0,\cdot,\omega)=u_0$), then
\begin{align}\lb{5.1}
\lim_{\eps\to 0} u^\eps(t,x+y_\eps,\omega)=\sup_{y\in x-t\calS} u_0(y) 
\end{align}
locally uniformly on $[0,\infty)\times\bbR^d$.

(ii)  If $\lim_{s\to\infty} \phi_H(s)=0$, then (i) holds with $\Lambda=\infty$ and with \eqref{5.1} replaced by 
\begin{align*}
\lim_{\eps\to0}\bbP\left[ \,\left|u^\eps(t,x+y_\eps,\omega)-\sup_{y\in x-t\calS} u_0(y) \right| \leq \delta\,\,\, \forall (t,x)\in [0,\delta^{-1})\times B_{\delta^{-1}}\right]=1
\end{align*}
for any $\delta>0$.
\end{theorem}

{\it Remarks.}
%
1.
Similarly to \eqref{0.2} in Theorem \ref{T.1.0}, hypotheses \eqref{Gc4} and  $\nabla_x\cdot v\equiv 0$ guarantee positive spreading speed of reachable sets in all directions \cite{BIN}.
\smallskip

2. Again, in (ii) we also have that $\limsup_{\eps\to0}u^\eps(t,x+y_\eps,\omega)\le \sup_{y\in x-t\calS} u_0(y)$ locally uniformly on $[0,\infty)\times\bbR^d$ (this is analogous to the proof of Remark 6 after Theorem \ref{T.1.0}).
\smallskip

\smallskip

{\bf Organization of the Paper and Acknowledgements.}
In Section \ref{s.2} we state a subadditive theorem from \cite{ZhaZla4}.  We then prove the two parts of Theorem \ref{T.1.0} in Sections \ref{s.3} and \ref{s.4}, and show how to extend these arguments to the case of Theorem \ref{T.4.0} in Section \ref{s.5}.

We thank 
Jessica Lin and Hung Tran
for useful discussions and pointers to literature.  YPZ acknowledges partial support by an AMS-Simons Travel Grant.  AZ acknowledges partial support by  NSF grant DMS-1900943 and by a Simons Fellowship.

\section{A Subadditive Theorem in Time-Dependent Environments} \lb{s.2}

In this section we provide for the convenience of the reader a new non-autonomous subadditive theorem, Theorem 1.2 in \cite{ZhaZla4} (also Remark 3 following it), that is a crucial ingredient in the proofs of our main results.  
Specifically, it will be used in the proofs of Lemmas \ref{L.2.4} and \ref{L.3.4} below.
    
\begin{theorem} \lb{T.1.1}   
Let $(\Omega,\bbP,\calF)$ be a probability space, and $\{\calF^{\pm}_t\}_{t\geq 0}$  two filtrations such that 
\[
\calF^{-}_s\subseteq\calF^{-}_{t}\subseteq \calF  \qquad\text{and}\qquad  \calF\supseteq \calF^{+}_{s}\supseteq\calF^{+}_{t}
\]
for all $t\geq s\geq 0$.
For any $t\ge 0$ and integers $n>m\geq 0$, let $X_{m,n}^t:\Omega\to [0,\infty)$ be a random variable.
Let there be $C\ge 0$ such that the following statements hold for all such $t,m,n$.
    \begin{enumerate}[(1)]
        \item $X_{m,n}^t \leq X_{m,k}^t+X_{k,n}^{t+X_{m,k}^t}$ for all $k\in\{m+1,\dots,n-1\}$;
        
        \smallskip
        
        \item $X_{0,1}^0\le C$;
        
        \smallskip
        
        \item the joint distribution of $\{X_{m,m+1}^t, X_{m,m+2}^t, \dots \}$ is independent of $(t,m)$;
        
        \smallskip
        
        \item   $X_{m,n}^t$ is $\calF^+_t$-measurable, and 
        $\{\omega\in\Omega\,|\, X_{m,n}^t(\omega)\le s\}\in \calF^{-}_{t+s}$ for any $s\ge 0$; 
        
         \smallskip
        
        \item  For some $\alpha>0$ we have $\lim_{s\to\infty}s^\alpha \phi(s)=0$, where 
\[
\phi(s):=\sup\left\{\left| \bbP[F|E] -\bbP[F]\right| \,\big|\,\, t\ge 0 \,\,\&\,\, (E,F)\in \calF_t^- \times \calF_{t+s}^+ \,\,\&\,\, \bbP[E]>0 \right\}.
\]
        
        \smallskip

        \item $X_{m,n}^t \leq X_{m,n}^{t+s}+s$ for all $s\in [C,C+c]$, with some $c>0$.
    \end{enumerate}
Then 
\beq\lb{11.1}
\lim_{n\to \infty}\frac{X_{0,n}^0 }{n}= \lim_{n\to \infty}\frac{\bbE\left[X_{0,n}^0\right]}{n}
\qquad\text{ almost surely.}
\eeq
Moreover, if in (5) we only have $\lim_{s\to\infty} \phi(s)=0$, then \eqref{11.1} holds in probability, as well as
\beq\lb{11.2}
\liminf_{n\to \infty}\frac{X_{0,n}^0 }{n}\ge \lim_{n\to \infty}\frac{\bbE\left[X_{0,n}^0\right]}{n}
\qquad\text{ almost surely.}
\eeq
\end{theorem}

%

\section{Proof of Theorem \ref{T.1.0}(i)}\lb{s.3}

We will consider general initial times $t_0\in\bbR$, and it will be convenient  to rewrite \eqref{1.0} as
\beq\lb{1.1}
u_t=\calL_{\Upsilon(t_0,0)(\omega)} u+f(t_0+t,x,u,\omega)
\eeq
(recall \eqref{0.3}), so that the solutions we will consider will always be defined on $\bbR^+\times\bbR^d$.  Then
for each $(t_0,x_0,\omega)\in\bbR^{d+1}\times\Omega$, let  $u(\cdot,\cdot,\omega;{t_0},{x_0})$ be the solution to  \eqref{1.1}  satisfying the initial condition $u(0,\cdot,\omega;t_0,x_0):=\frac 12\chi_{B_1(x_0)}$.  For any $\theta\in (0,1)$, we let
\[
\Gamma_{\theta}(t,\omega;t_0,x_0):=\left\{x\in\bbR^d\,\big|\, u(t,x,\omega;{t_0},{x_0})\geq \theta\right\}
\]
be its $\theta$-super-level set at time $t\geq 0$.
Let us also denote $\Gamma_\theta(t,\omega):=\Gamma_\theta(t,\omega;0,0)$.

It is proved in \cite{ZlaKPPspread} (see the proof of Theorem 1.4 there) that a uniform {\it hair-trigger effect} holds under the hypotheses of Theorem \ref{T.1.0}.  Specifically, for any fixed $\theta\in(0,1)$ we have that any solution to \eqref{1.1} with $u(0,\cdot)\ge \theta \chi_{B_1(0)}$ converges locally uniformly on $\bbR^d$ to 1 as $t\to\infty$, and this convergence is uniform in all $(A,b,f)$ (as well as in all $(t_0,\omega)\in\bbR\times \Omega$)  that satisfy the hypotheses of Theorem \ref{T.1.0} uniformly --- that is, with the same  
\[
\gamma\in \left(0, \min \left\{ \lambda ,\, \|A\|_\infty^{-1},  \,\|f_u(\cdot,\cdot,0,\cdot)\|_{L^{\infty}}^{-1}, \,4 \lambda \inf_{(t,x,\omega)\in\bbR^{d+1} \times\Omega} f_u(t,x,0,\omega) - \|b\|_{L^\infty}^2   \right\} \right],
\]
the same Lipschitz lower bound $f_0:(0,1)\to(0,\infty)$ on $\tilde f(u):=\inf_{(t,x,\omega)\in\bbR^{d+1}\times\Omega}f(t,x,u,\omega)$, and the $\sup$ in \eqref{0.5} bounded above by the same $\psi(u)$ with $\lim_{u\to 0} \psi(u)=0$.
 Of course, the uniformity in $\omega$ then also extends the uniform convergence to any spatial shift of the initial datum, after accounting for the corresponding shift in the solution (because shifting the medium by $z\in\bbR^{d}$ simply amounts to changing $\omega$ to $\Upsilon_{(0,z)}(\omega)$).  Note that bootstrapping this claim then yields at least ballistic  spreading  of each super-level set in all directions (with the same positive lower bound on the spreading speed for all the super-level sets, because such lower bound for the $\frac 12$-super-level set also applies to all other $\theta\in(0,1)$ due to the hair-trigger effect).  On the other hand, a finite upper bound on the spreading speeds follows from $e^{at-(x-x_0)\cdot e}$ being a super-solution to \eqref{1.1} for any $(e,x_0)\in\bbS^{d-1}\times \bbR^d$, provided $a$ is large enough (depending only on $\gamma$ above --- see, e.g., the proof of Theorem 2.1 in \cite{ZlaKPPspread}).

%
%
%
%
%

In particular, there is $M\ge 1$ (which depends only on $\gamma,f_0,\psi$ above)  such that under the hypotheses of Theorem \ref{T.1.0} we have for all $(t_0,x_0,\omega)\in\bbR^{d+1}\times\Omega$,
\beq\lb{2.12}
B_{M^{-1}t}(x_0)\subseteq \Gamma_{1/2}(t,\omega;t_0,x_0)\subseteq B_{{M}t}(x_0)\qquad \text{ when $t\geq M$}.
\eeq
This immediately yields $u(s,\cdot,\omega;t_0,x_0)\geq u(0,\cdot,\omega;t_0+s,z)$ for any $(t_0,x_0,z,\omega)\in\bbR^{2d+1}\times\Omega$ and $s\geq {M}(|z-x_0|+1)$. Hence the comparison principle shows that for any $t\ge 0$, 
\beq\lb{2.31}
\Gamma_{1/2}(t,\omega;t_0+s,z )\subseteq\Gamma_{1/2}(t+s,\omega;t_0,x_0 )\qquad\text{when }s\geq {M}(|z-x_0|+1).
\eeq
Parabolic Harnack inequality \cite[Corollary 7.42]{Lie} shows that there is $\theta>0$ (depending only on $\gamma$) such that if $x\in \Gamma_{1/2}(t+s,\omega;t_0,x_0 )$ and $t+s\ge 1$, then $u(t+s+1,\cdot,\omega;t_0,x_0)\ge\theta\chi_{B_1(x)}$.  Hence if we increase $M$ to three times the maximum of $M$ and a time $\tau\ge 1$ such that under the hypotheses of Theorem \ref{T.1.0}, any solution to \eqref{1.1} with $u(0,\cdot)\ge \theta\chi_{B_1(0)}$ satisfies  $u(\tau,\cdot)\ge \frac 12\chi_{B_1(0)}$ (such $\tau$ exists due to the hair-trigger effect), then from \eqref{2.31} we obtain with any $s'\ge M$,
\beq\lb{2.31'}
B_{M^{-1}s'} \left( \Gamma_{1/2}(t,\omega;t_0+s,z ) \right)\subseteq\Gamma_{1/2}(t+s+s',\omega;t_0,x_0 )\qquad\text{when }s\geq {M}(|z-x_0|+1)
\eeq
by the hair-trigger effect and another usage of \eqref{2.12}.

Next let  the travel time to a point $x\in\bbR^d$, when starting at $(t_0,x_0)\in \bbR^{d+1}$, be
\[
\tau^{t_0}({x_0,x},\omega):=\inf\left\{t\geq 0\,\big|\, B_1(x)\subseteq \Gamma_{1/2}(t,\omega;t_0,x_0) \right\}.
\]
The comparison principle yields a space-time subadditivity property for these times, namely 
\beq\lb{2.1}
\tau^{t_0}({x_0,x},\omega)\leq \tau^{t_0}({x_0,z},\omega)+\tau^{t_0+\tau^{t_0}({x_0,z},\omega)}({z,x},\omega)
\eeq
for any $(z,\omega)\in\bbR^d \times \Omega$.
Due to \eqref{2.12}, we also have 
\beq\lb{2.0}
\tau^{t_0}({x_0,x},\omega)\leq {M}(|x-x_0|+1)
\eeq
for all $(t_0,x_0,x,\omega)\in \bbR^{2d+1} \times \Omega$.
Combining this with \eqref{2.1} yields
\beq\lb{2.7}
\tau^{t_0}({x_0,x},\omega)\leq \tau^{t_0}({x_0,z},\omega)+{M}(|x-z|+1)
\eeq
for all $(t_0,x_0,x,z,\omega)\in \bbR^{3d+1} \times \Omega$.
Finally, from \eqref{2.31}  we have 
\beq\lb{2.2}
\tau^{t_0}({x_0,x},\omega)\leq \tau^{t_0+t}({z,x},\omega)+t\qquad\text{ when $t\geq M(|z-x_0|+1)$.}
\eeq

We can now prove Theorem \ref{T.1.0}(i).  
We will first assume that $H:=(A,b,f)$ is space-time stationary (with $f$ understood as a $C([0,1])$-valued function on $\bbR^{d+1}\times\Omega$), rather than just $(A,b,f_u(\cdot,\cdot,0,\cdot))$.  
We will also denote $\calF_t^\pm:=\calF_t^\pm(H)$ for simplicity (see Definition \ref{D.2.3}).

We start with a lemma that shows that the travel times in any fixed direction are asymptotically linear.

\begin{lemma}\lb{L.2.4}
Assume the hypotheses of Theorem \ref{T.1.0}(i) but with $H:=(A,b,f)$ being space-time stationary. Then for  each $e\in\bbS^{d-1}$ there are $\Omega_e\subseteq \Omega$ and
\beq\lb{2.5}
\bar{\tau}(e)\in\left[M^{-1}, M \right],
\eeq
with $M$ from \eqref{2.12}, such that $\bbP[\Omega_e]=1$ and for each  $(t_0,x_0,\omega)\in \bbR^{d+1}\times\Omega_e$ we have
\beq\lb{2.4}
\lim_{r\to\infty}\frac{\tau^{t_0}({x_0,x_0+re},\omega)}{r}=\lim_{r\to\infty}\frac{\bbE\left[ \tau^{t_0}({x_0,x_0+re},\cdot)\right]}{r}=\bar{\tau}(e).
\eeq
Moreover, for any $e,e'\in\bbS^{d-1}$ we have
\beq\lb{2.6}
\max\left\{ |\bar{\tau}(e)-\bar{\tau}(e')|,\left|\frac1{\bar{\tau}(e)}-\frac1{\bar{\tau}(e')}\right|\right\}\leq M^3|e-e'|.
\eeq
\end{lemma}

\begin{proof}
For any $e\in\bbS^{d-1}$, the hypotheses of Theorem \ref{T.1.1} hold with $X_{m,n}^t :=\tau^{t}(me,ne,\cdot)$.  Indeed:  (1) follows from \eqref{2.1}; (2) from \eqref{2.0}; (3) from space-time stationarity of $H$; (4) from the definition of $\calF_t^\pm$; (5) from the hypothesis;
and  (6) with $(C,c):=(M,\infty)$ from \eqref{2.2} with $z=x_0$.  Hence Theorem \ref{T.1.1} and \eqref{2.7} yield \eqref{2.4} with $(t_0,x_0)=(0,0)$ for almost all $\omega\in\Omega$, with \eqref{2.5} following from \eqref{2.12}.
Space-time stationarity of $H$ yields some full measure set $\Omega_e\subseteq\Omega$ such that \eqref{2.4} holds for all $(t_0,x_0)\in\bbZ^{d+1}$ when $\omega\in\Omega_e$, and this then extends to all $(t_0,x_0)\in\bbR^{d+1}$ by \eqref{2.2}.

Next, \eqref{2.1} and \eqref{2.0} yield for any $e,e'\in\bbS^{d-1}$,
\begin{align*}
\tau^0(0,ne,\omega)\leq \tau^0(0,ne',\omega)+\tau^{\tau^{0}(0,ne',\omega)}(ne,ne',\omega)\leq \tau^0(0,ne',\omega)+{M}(n|e-e'|+1).
\end{align*}
After dividing this by $n$ and taking $n\to\infty$, we obtain 
$\bar{\tau}(e)\leq \bar{\tau}(e')+{M}|e-e'|$.
This and $\bar{\tau}\geq M^{-1}$ yield \eqref{2.6}.
\end{proof}

Next, we show that solutions to \eqref{1.1} with localized initial data asymptotically approximate characteristic functions of a ballistically expanding deterministic Wulff shape
\beq\lb{2.3}
\calS:=\{se\,\big|\,e\in\bbS^{d-1}\text{ and }s\in [0,w(e))\}, 
\eeq
where the deterministic spreading speed in direction $e\in\bbS^{d-1}$ is
\beq\lb{2.3'}
w(e):=\bar{\tau}(e)^{-1}\in [{M}^{-1},{M}].
\eeq
Note that $\calS$ is bounded and open due to \eqref{2.6}.


\begin{theorem}\lb{T.2.5}
Under the hypotheses of Lemma \ref{L.2.4}, $\calS$ from \eqref{2.3} is convex, and it is a strong deterministic Wulff shape for \eqref{1.0}.  The latter means that for almost all $\omega\in\Omega$, 
\beq \lb{11.3}
\begin{aligned} 
    \lim_{t\to\infty}\inf_{|x_0|\leq\Lambda t} \inf_{x\in (1-\delta)t\calS}u(t,x_0+x,\omega;0,x_0) &=1,\\
\lim_{t\to\infty}\sup_{|x_0|\leq \Lambda t} \sup_{x\not\in (1+\delta)t\calS}u(t,x_0+x,\omega;0,x_0) &=0
\end{aligned}
\eeq
hold for each $\delta\in (0,1)$ and $\Lambda\ge 0$. 
\end{theorem}

%

\begin{proof}
Convexity of $\calS$ will be proved in Theorem \ref{T.2.6} under slightly more general hypotheses.

For each $e\in \bbS^{d-1}$, let $\Omega_e$ be the set from Lemma \ref{L.2.4}. Let $Q$ be a countable dense subset of $\bbS^{d-1}$ and define $\Omega':=\cap_{e\in Q}\Omega_e$ (so $\bbP[\Omega']=1$).  Now fix any $\delta\in (0,1)$ and $\omega\in\Omega'$.  We will first show that there is $C_{\delta,\omega}>0$ such that for all $t\geq C_{\delta,\omega}$,
\beq\lb{2.10}
(1-\delta)t\calS\subseteq \Gamma_{1/2}(t,\omega)
\subseteq (1+\delta) t\calS.
\eeq

Let $\eps:=\frac{\delta}{3M}$
and let $e_1,\dots,e_N\in \calS\setminus\{0\}$ be such that $\frac{e_i}{|e_i|}\in Q$ and  $\calS\subseteq \bigcup_{i=1}^NB_\eps(e_i)$. Hence for any $t\ge 0$ and $v\in t\calS$, there is $i\in\{1,\dots,N\}$ such that $|v-te_i|\leq t\eps$. Then \eqref{2.7} shows that
\beq\lb{2.18}
\left| \tau^0(0,v,\omega)-\tau^0(0,te_i,\omega)\right|\leq {M}(t\eps+1).
\eeq
By Lemma \ref{L.2.4}, for all large enough $t$ we have 
\beq\lb{2.13}
\sup_{i\in\{1,\cdots,N\}}  \left|\frac{\tau^0(0,te_i,\omega)}{t|e_i|}-\frac{1}{w({e_i}{|e_i|^{-1}})}\right|\leq \eps
\eeq
Using \eqref{2.18}, \eqref{2.13}, and $|e_i|\leq w(\frac{e_i}{|e_i|})\leq {M}$ (by \eqref{2.3} and \eqref{2.3'}), for all large $t$ we obtain 
\[
\sup_{v\in t\calS}\tau^0(0,v,\omega)\leq \max_{i\in\{1,\cdots,N\}} \tau^0(0,te_i,\omega)+{M}(t\eps+1)\leq t+{M}(2t\eps+1).
\]
Hence 
\beq\lb{2.51}
t\calS\subseteq  \Gamma_{1/2} (t+{M}(2t\eps+1),\omega)
\eeq
holds for all large enough $t$. Since $2M\eps< \delta$,  the first inclusion in \eqref{2.10} follows.

Next let $e_1',\dots,e_{N'}'\subseteq\bbR^d\setminus\calS$ be such that $\frac{e_i'}{|e_i'|}\in Q$ and $B_{{M}}(0)\backslash \calS \subseteq \bigcup_{i=1}^{N'} B_\eps(e_i') $. Note that 
\beq\lb{2.14}
v\notin \Gamma_{1/2}(t,\omega) \qquad \text{ whenever $t\geq M$ and $v\in B_{{M}t}(0)^c$}
\eeq
due to \eqref{2.12}.
For each $v\in B_{{M}t}(0)\backslash t\calS$, there is $e_i'$ such that $|v-te_i'|\leq t\eps$ and then
\[
 \tau^0(0,v,\omega)\geq \tau^0(0,te_i',\omega)- {M}(t\eps+1).
\]
by \eqref{2.7}. Moreover, since now $w(\frac{e_i'}{|e_i'|}) \le |e_i'| \le {M}$ and \eqref{2.13} holds with $e_i'$ and $N'$ in place of $e_i$ and $N$, we obtain
\[
\inf_{v\in B_{{M}t}(0)\backslash t\calS}\tau^0(0,v,\omega)\geq \min_{i\in\{1, \cdots,N'\}} \tau^0(0,te_i',\omega)-{M}(t\eps+1)\geq t-{M}(2t\eps+1).
\]
This and \eqref{2.14} yield
$\Gamma_{1/2}(t-2{M}(t\eps+1),\omega)\subseteq  t\calS$ for all large enough $t$, so the second inclusion in \eqref{2.10} again follows by $2M\eps< \delta$.

We next want to upgrade \eqref{2.10} to the claim that for almost all $\omega$ we have for any $\Lambda\ge 1$ and  $\delta\in(0,1)$,
\beq\lb{2.10'}
x_0+(1-\delta) t\calS\subseteq\Gamma_{1/2}(t,\omega;0,x_0) 
\subseteq x_0+(1+\delta) t\calS
\eeq
for all large enough $t$ (depending on $\omega,\delta,\Lambda$) and all $x_0\in B_{\Lambda t}(0)$. 
This will finish the proof because parabolic Harnack inequality and hair-trigger effect show that for each $\theta\in (0,1)$, there is $C_\theta>0$ such that for all $(t_0,x_0,\omega)\in\bbR^{d+1}\times \Omega$ and $t\geq C_\theta+1$,
\[
\Gamma_{1/2}(t-C_\theta,\omega;t_0,x_0)\subseteq  \Gamma_\theta(t,\omega;t_0,x_0)\subseteq \Gamma_{1/2}(t+C_\theta,\omega;t_0,x_0).
\]

It therefore remains to show \eqref{2.10'}.  Fix any $\Lambda\ge 1$ and
\beq\lb{2.45}
\delta\in\left( 0,  ({26M\Lambda})^{-1} \right).
\eeq
By \eqref{2.10} and Egorov's Theorem, 
there are $\tau_{\delta}>0$ and $D_{\delta}\subseteq\Omega$ with $\bbP[D_{\delta}]\geq 1-\delta^{d+1}$ such that for each $\omega\in D_{\delta}$ and $t\geq \tau_{\delta}$,
\beq\lb{2.30}
(1-\delta) t\calS\subseteq \Gamma_{1/2}(t,\omega) 
\subseteq (1+\delta)  t\calS.
\eeq
It is clear that we can in fact pick $D_{\delta}$ from the $\sigma$-algebra generated by $\bigcup_{i\in\bbN} (\calF_0^+\cap \calF_i^-)$.
Let $\calF'$ be the $\sigma$-algebra generated by $\bigcup_{i\in\bbN}(\calF_{-i}^+\cap \calF_i^-)$ (or just replace $\calF$ by $\calF'$ from the very start) and apply Wiener's ergodic theorem (see, e.g.,  \cite[Theorems 2 and 3]{Bec}) with the group of transformations $\{\Upsilon_{(s,y)}\}_{(s,y)\in \bbR^{d+1}}$ on the probability space $(\Omega,\calF',\bbP)$.  It shows that there is $\Omega_{\delta}\in\calF'$ with $\bbP[\Omega_{\delta}]=1$ such that the following holds, with
\[
\varphi_{\delta,r}(\omega):=\frac{1}{|{\calB}_{ r}|}\int_{{\calB}_{ r}} \chi_{D_{\delta}} \left( \Upsilon_{(s,y)}(\omega) \right) dsdy
\]
and $\calB_r\subseteq\bbR^{d+1}$ the space-time ball of radius $r>0$ centered at the origin.
 The limit 
\[
\varphi_{\delta}(\omega):= \lim_{r\to\infty} \varphi_{\delta,r}(\omega) \, \in[0,1]
\]
(which is $\calF'$-measurable because $\varphi_{\delta,r}$ is measurable with respect to the $\sigma$-algebra generated by $\bigcup_{i\in\bbN} (\calF_{-r}^+\cap \calF_i^-)$) exists for each $\omega\in \Omega_{\delta}$, is invariant under the transformations $\{\Upsilon_{(s,y)}\}_{(s,y)\in \bbR^{d+1}}$, and satisfies
\[
\bbE[\varphi_{\delta}(\cdot)]=\bbE[\chi_{D_{\delta}}(\cdot)]=\bbP[D_{\delta}].
\]

Next we claim that $\varphi_{\delta}$ is a constant almost everywhere on $\Omega$.  If not, then there are $k\in\bbN$, $c>0$, and $A_1,A_2\in \calF_{-k}^+\cap\calF_{k}^-$ with $\bbP[A_1]\bbP[A_2]> 0$ such that
\beq\lb{2.41}
\left|\frac{1}{\bbP[A_1]}\bbE[\varphi_{\delta}(\cdot)\chi_{A_1}(\cdot)]- \frac{1}{\bbP[A_2]}\bbE[\varphi_{\delta}(\cdot)\chi_{A_2}(\cdot)]\right|\geq c.
\eeq
Fix $C'>0$ and note that $\varphi_{\delta,r}(\Upsilon_{(r+k+C',0)}(\cdot))$ is $\calF_{k+C'}^+$-measurable.  Hence the definition of $\phi_H$, the fact that $A_j\in\calF^-_k$, and $0\le \varphi_{\delta,r} \le 1$ yield for $j=1,2$,
\begin{align*}
\bbE[\varphi_{\delta,r}(\Upsilon_{(r+k+C',0)}(\cdot))\chi_{A_j}(\cdot)] & =\int_0^1 \bbP\left[\varphi_{\delta,r}(\Upsilon_{(r+k+C',0)}(\cdot))>\mu\,\,\&\,\,\chi_{A_j}=1\right]  d\mu\\
& \leq \int_0^1 \left(\bbP\left[\varphi_{\delta,r}(\Upsilon_{(r+k+C',0)}(\cdot))>\mu\right]+\phi_H(C')\right)\bbP\left[ A_j\right]  d\mu\\
& \leq\bbE[\varphi_{\delta,r}(\Upsilon_{(r+k+C',0)}(\cdot))]\bbP\left[ A_j\right]+\phi_H(C').
\end{align*}
Similarly,
\[
\bbE[\varphi_{\delta,r}(\Upsilon_{(r+k+C',0)}(\cdot))\chi_{A_j}(\cdot)]\geq\bbE[\varphi_{\delta,r}(\Upsilon_{(r+k+C',0)}(\cdot))]\bbP\left[ A_j\right]-\phi_H(C').
\]
Since $\lim_{s\to\infty}\phi_H(s)=0$ and $\bbP[A_j]>0$ for $j=1,2$, taking  sufficiently large  $C'$ yields
\[
\left|\frac{1}{\bbP\left[ A_j\right]}\bbE[\varphi_{\delta,r}(\Upsilon_{(r+k+C',0)}(\cdot))\chi_{A_j}(\cdot)] -\bbE[\varphi_{\delta,r}(\Upsilon_{(r+k+C',0)}(\cdot))]\right|\leq \frac{c}{4}.
\]
Since $\varphi_{\delta,r}\to\varphi_{\delta}$ almost surely as $r\to\infty$, and $0\le \varphi_{\delta,r} \le 1$, we thus obtain
\[
\left|\frac{1}{\bbP[A_j]}\bbE[\varphi_{\delta}(\Upsilon_{(r+k+C',0)}(\cdot))\chi_{A_j}(\cdot)]-\bbE[\varphi_{\delta,r}(\Upsilon_{(r+k+C',0)}(\cdot))]\right|<\frac{c}{2}
\]
for all large enough $r$ and $j=1,2$.
However, since $\varphi_{\delta}\circ \Upsilon_{(r+k+C',0)} = \varphi_{\delta}$, this contradicts 
\eqref{2.41}.
Thus we see that $\varphi_{\delta}(\omega) = \bbP[D_{\delta}]$ for almost all $\omega\in\Omega$.

This means that there is $\Omega_{\delta}'\subseteq\Omega_{\delta}$ with $\bbP[\Omega_{\delta}']=1$ such that  for each $\omega\in \Omega_{\delta}'$ we have
\[
\lim_{r\to\infty} \frac{1}{|{\calB}_{ r}|}\int_{{\calB}_{ r}} \chi_{D_{\delta}}(\Upsilon_{(s,y)}(\omega))dsdy=\bbP[D_{\delta}]\geq 1-\delta^{d+1}.
\]
Thus there is $t_{\omega,\delta,\Lambda}\ge \max\{\tau_\delta, \frac 1{\delta}\} $ such that for all $t\geq t_{\omega,\delta,\Lambda}$ we have
\[
\left|\left\{(s,z)\in  \calB_{2\Lambda t}\,|\,\Upsilon_{(s,z)}(\omega)\notin D_{\delta}\right\}\right|\leq 2\delta^{d+1}\left|\calB_{2\Lambda t}\right|.
\]
For any  $t\geq t_{\omega,\delta,\Lambda}$, let $C_t:=M(2\delta\Lambda t+1)\le 3M\delta\Lambda t$ (because $t\ge \frac 1 {\delta \Lambda}$).
Then for any $x_0\in B_{\Lambda t}(0)$, there are 
\[
(s_\pm,z_\pm)\in [C_t,C_t+8\delta\Lambda t] \times B_{2\delta \Lambda t}(x_0) \subseteq \calB_{2\Lambda t}
\]
satisfying $\Upsilon_{(\pm s_\pm,z_\pm)}(\omega)\in D_{\delta}$ (note that $(-2\delta \Lambda t,2\delta \Lambda t) \times B_{2\delta \Lambda t}(0) \supseteq \calB_{2\delta \Lambda t}$, while \eqref{2.45} implies  $C_t+10\delta\Lambda t \le 13M\delta \Lambda t \le  \Lambda t$). 
Now let $c_{\delta,\Lambda}:=13M\delta\Lambda$ ($\le \frac 12$ by \eqref{2.45}).
Since we have $s_\pm \geq C_t\geq M(|z_\pm-x_0|+1)$ and $2M\delta\Lambda t\ge M$, as well as
\[
s_\pm + 2M\delta\Lambda t \le C_t+10M\delta\Lambda t \le c_{\delta,\Lambda} t,
\]
from \eqref{2.31'}, \eqref{2.30}, and $\Upsilon_{(\pm s_\pm,z_\pm)}(\omega) \in D_{\delta}$ we obtain
\begin{align*}
\Gamma_{1/2}(t,\omega;0,x_0)-x_0  & \subseteq \Gamma_{1/2}(t+s_-+2M\delta\Lambda t,\omega; -s_-,z_-)-z_- 
\\ & = \Gamma_{1/2}(t+s_- +2M\delta\Lambda t ,\Upsilon_{(-s_-,z_-)}(\omega))  \subseteq (1+\delta)(1+c_{\delta,\Lambda}) t\calS
\end{align*}
and
\begin{align*}
\Gamma_{1/2}(t,\omega;0,x_0) -x_0 & \supseteq \Gamma_{1/2}(t-s_+-2M\delta\Lambda t,\omega;s_+,z_+) -z_+
\\ & = \Gamma_{1/2}(t-s_+-2M\delta\Lambda t,\Upsilon_{(s_+,z_+)}(\omega))  \supseteq (1-\delta)(1-c_{\delta,\Lambda}) t\calS
\end{align*}
for any $\omega\in\Omega'_\delta$ and $t\geq t_{\omega,\delta,\Lambda}$.
Since $\lim_{\delta\to 0}c_{\delta,\Lambda}=0$ for each $\Lambda\ge 1$, this shows that for each $\omega\in \Omega'':= \bigcap_{L\in(26M,\infty)\cap\bbN} \Omega'_{1/L}$ (so $\bbP[\Omega'']=1$) and  $(\delta,\Lambda)\in(0,1)\times [1,\infty)$,   we indeed have \eqref{2.10'} 
 when $t$ is large enough and $x_0\in B_{\Lambda t}(0)$.
\end{proof}

If now only $H:=(A,b,f_u(\cdot,\cdot,0,\cdot))$ is space-time stationary, we let 
\beq\lb{0.9}
f'(t,x,u,\omega):=f_u(t,x,0,\omega) \min\{u,1-u\},
\eeq
so that Lemma \ref{L.2.4} and  Theorem \ref{T.2.5} apply to \eqref{1.0} with $f'$ in place of $f$.  One can then use {\it virtual linearity} of \eqref{1.0}  with KPP $f$, expressed in Theorem 1.2 in \cite{ZlaKPPspread}, to conclude Theorem~\ref{T.1.0}(i) for $f'$, as well as to show that the leading order solution dynamics (as $t\to\infty$) of \eqref{1.0} with the two reactions coincide (this then proves Theorem \ref{T.1.0}(i) for $f$, and also that $\calS$ only depends on $(A,b,f_u(\cdot,\cdot,0,\cdot))$).  This argument is identical to that in the last third of the proof of Theorem 1.4 in \cite{ZlaKPPspread}, which concerns time-periodic and spatially stationary ergodic $(A,b,f)$, because it only uses Theorem 1.2 in \cite{ZlaKPPspread} (which applies to general space-time-dependent $(A,b,f)$) and Theorem \ref{T.2.5}.
We therefore skip it; the next section contains an analog in the setting of Theorem \ref{T.1.0}(ii).  We also note that this  shows that Lemma \ref{L.2.4} and  Theorem \ref{T.2.5} hold even if only $(A,b,f_u(\cdot,\cdot,0,\cdot))$  (and not $(A,b,f)$) is space-time stationary.

\section{Proof of Theorem \ref{T.1.0}(ii)} \lb{s.4}

The arguments from the start of the previous section (prior to Lemma \ref{L.2.4}) also apply here, and we will again consider \eqref{1.1} as well as $u(t,x ,\omega;{t_0},{x_0})$, $\Gamma_\theta(t,\omega;{t_0},{x_0})$, and $\tau^t_0(x_0,x ,\omega)$ as above.  We will also again first assume that $H:=(A,b,f)$ is space-time stationary, and denote $\calF_t^\pm:= \calF_t^\pm(H)$.
We now have the following analogs of Lemma \ref{L.2.4} and Theorem \ref{T.2.5}.

\begin{lemma}\lb{L.3.4}
Assume the  hypotheses of Lemma \ref{L.2.4}, but with $\lim_{s\to\infty}\phi_H(s)=0$ for $H:=(A,b,f)$ instead of $\lim_{s\to\infty}s^\alpha \phi_H(s)=0$. Then for each $e\in\bbS^{d-1}$, there is $\bar{\tau}(e)$ satisfying \eqref{2.5} and \eqref{2.6} such that for each $(t_0,x_0)\in \bbR^{d+1}$ we have
\beq\lb{2.15}
\lim_{r\to\infty}\frac{\tau^{t_0}(x_0,x_0+re,\cdot)}{r} =\lim_{r\to \infty}\frac{\bbE\left[\tau^{t_0}(x_0,x_0+re,\cdot)\right]}{r}=\bar{\tau}(e)\qquad\text{ in probability}.
\eeq
\end{lemma}

\begin{proof}
Same as Lemma \ref{L.2.4},  using the convergence in probability claim in Theorem \ref{T.1.1}. 
\end{proof}


\begin{theorem}\lb{T.2.6}
Under the hypotheses of Lemma \ref{L.3.4}, $\calS$ from \eqref{2.3} with $w$ from \eqref{2.3'} is convex, and
it is a strong Wulff shape for \eqref{1.0} in probability.  The latter means that
\beq\lb{2.17}
\begin{aligned}
\lim_{t\to\infty}\bbP\left[(1-\delta) s\calS\subseteq\Gamma_{\theta}(s,\cdot;0,x_0)-x_0\subseteq (1+\delta) s\calS\,\,\forall (s,x_0)\in  [\delta t ,\delta^{-1}t] \times B_{\Lambda t} \right]=1
\end{aligned}
\eeq
holds  for each $\delta,\theta\in (0,1)$ and $\Lambda \geq 0$.
\end{theorem}

\begin{proof}

Let $e_1,e_2\in\bbS^{d-1}$ be arbitrary with $e_2\neq -e_1$, and let $e':=\frac{e_1+e_2}{|e_1+e_2|}$.  From \eqref{2.1} and $|e_1+e_2|e'=e_1+e_2$ and we obtain for each $r>0$ and $\omega\in\Omega$,
\beq\lb{2.8}
\tau^0(0,|e_1+e_2|re',\omega)\leq \tau^0(0,re_1,\omega)+\tau^{\tau^0(0,re_1,\omega)}(re_1,r(e_1+e_2),\omega).
\eeq
Then \eqref{2.15} shows that for each $\eps>0$ and any large enough $r$,
there is  $\omega_{r,\eps}\in\Omega$ such that
\[
\max \left \{ |\tau^0(0,|e_1+e_2|re',\omega_{r,\eps})-|e_1+e_2|r\bar{\tau}(e')|,\, |\tau^0(0,re_1,\omega_{r,\eps})- r\bar{\tau}(e_1)| \right\}\leq r\eps,
\]
as well as (using also \eqref{2.2})
\[
\tau^{\tau^0(0,re_1,\omega_{r,\eps})}(re_1,r(e_1+e_2),\omega_{r,\eps})\leq \tau^{r\bar{\tau}(e_1)+r\eps}(re_1,r(e_1+e_2),\omega_{r,\eps})+2r\eps\leq r\bar{\tau}(re_2)+3r\eps.
\]
From these and \eqref{2.8} we obtain
\[
|e_1+e_2|\bar\tau(e')\leq \bar\tau(e_1)+\bar\tau(e_2)+5\eps,
\]
and after taking $\eps\to 0$ this becomes
\[
w(e')\geq \frac{w(e_1)w(e_2)}{w(e_1)+w(e_2)}|e_1+e_2|.
\]
The angle bisector theorem now shows that $\calS$ is convex.

It remains to prove \eqref{2.17}.  Let $\eps:=\frac{\delta^2}{10M(1+M\delta)}$, let $s_1,\ldots,s_N\in [\delta,\delta^{-1}]$ be such that $[\delta,\delta^{-1}]\subseteq \bigcup_{j=1}^NB_\eps(s_j)$,
and let $e_1,\ldots,e_N\in \calS\backslash\{0\}$ be such that $\calS\subseteq \bigcup_{i=1}^NB_\eps(e_i)$. Thus for any $t>0$, and any $s\in [\delta t, \delta^{-1}t]$ and $v\in  s\calS$, there are $i,j\in\{1,\ldots,N\}$ such that $|s-ts_j|\leq t\eps$ and $|v-se_i|\leq s\eps\leq \delta^{-1} t\eps$. It follows from \eqref{2.7} and $|e_i|\leq M$ that
\beq\lb{2.18'}
\left| \tau^0(0,v,\omega)-\tau^0(0,ts_je_i,\omega)\right|\leq {M}(|v-ts_je_i|+1)\leq M(\delta^{-1} t\eps+Mt\eps+1).
\eeq
By Lemma \ref{L.3.4}, we also have
\beq\lb{2.16}
\lim_{t\to\infty} \bbP\left [    \left|\frac{{\tau^0}(0,ts_je_i,\cdot)}{ts_j|e_i|}-\frac{{1}}{w({e_i}{|e_i|^{-1}})}\right|\geq \eps \right]=0
\eeq
for each $i,j\in\{1,\ldots,N\}$.
From \eqref{2.18'} and $s_j|e_i|\leq M\delta^{-1} $ we obtain
\begin{align*}
\bbP & \left[\sup_{v\in  s\calS}\tau^0(0,v,\cdot)\geq (s+t\eps)+ M\delta^{-1}t\eps +M(\delta^{-1} t\eps+Mt\eps+1)\,\text{ for some }s\in [\delta t, \delta^{-1}t]\right]\\
&\qquad\leq \bbP\left[\max_{i,j\in\{1,\cdots,N\}} \tau^0(0,ts_je_i,\cdot)\geq ts_j+M\delta^{-1} t\eps \right] \\
&\qquad\leq \sum_{i,j=1}^N\bbP\left[ \tau^0(0,ts_je_i,\cdot)\geq ts_j+ ts_j|e_i|\eps\right],
\end{align*}
which converges to $0$ as $t\to \infty$ by \eqref{2.16} and $|e_i|\le w(\frac{e_i}{|e_i|})$. This implies that
\beq\lb{2.19}
\lim_{t\to\infty}\bbP\left[ s\calS\subseteq\Gamma_{1/2}(s+2M(\delta^{-1}+M)t\eps +M,\cdot)\,\,\forall s\in[\delta t, \delta^{-1}t]\right]=1.
\eeq

Next, let $e_1',\ldots,e_{N'}'\in  \bbR^d\backslash \calS$ be such that $B_{M}(0)\backslash \calS \subseteq \bigcup_{i=1}^{N'} B_\eps(e_i') $. It is clear that \eqref{2.14} still holds, as does \eqref{2.16} with $e_i'$ and $N'$ in place of $e_i$ and $N$.
For any $t\geq M\delta^{-1}$, and any $s\in [\delta t, \delta^{-1}t]$ and $v\in B_{{M}s}(0)\backslash  s\calS$, there are $i\in\{1,\ldots,N'\}$ and $j\in\{1,\ldots,N\}$ such that
\eqref{2.18'} holds with $e_i'$ in place of $e_i$.
This, together with \eqref{2.14} and $s_j|e_i|\leq M\delta^{-1}$, yields 
\begin{align*}
\bbP & \left[\inf_{v\in ( s\calS)^c}\tau^0(0,v,\cdot)\leq s-t\eps - M\delta^{-1}t\eps -{M}(\delta^{-1} t\eps+Mt\eps+1)\,\text{ for some }s\in [\delta t, \delta^{-1}t]\right]\\
&\qquad\leq \bbP\left[\min_{i\in\{1,\cdots,N'\}\,\&\, j\in\{1,\cdots,N\}} \tau^0(0,ts_j e_i',\cdot)\leq ts_j-M\delta^{-1} t\eps \right]\\
&\qquad\leq \sum_{i=1}^{N'}\sum_{j=1 }^N \bbP\left[\tau^0(0,ts_je_i',\cdot)\leq ts_j-ts_j|e_i|\eps\right],
\end{align*}
which converges to $0$ as $t\to\infty$ by \eqref{2.16} and $|e_i'|\geq w(\frac{e_i'}{|e_i'|})$.
Therefore we get
\[
\lim_{t\to\infty}\bbP\left[\Gamma_{1/2}(s-2M(\delta^{-1}+M)t\eps -M,\cdot)\subseteq  s\calS\,\, \forall s\in [\delta t, \delta^{-1}t]\right]=1.
\]
Since $\eps<\frac{\delta}{2M(\delta^{-1}+M)}$, this and \eqref{2.19} yield \eqref{2.17} with $(\theta,\Lambda)=(\frac12,0)$.

Let us now extend this to the general case. Fix any $\delta\in(0,1)$ and again let $\eps:=\frac{\delta^2}{10M(1+M\delta)}$.  Stationarity of $(A,b,f)$ and \eqref{2.17} with $(\theta,\Lambda)=(\frac12,0)$ 
show that for each ${\sigma}\in(0,1)$, there is $C_{{\sigma}}\ge\frac 1\eps$ such that for any $(t_0,z)\in\bbR^{d+1}$ and $t\geq C_{{\sigma}}$,
\beq\lb{2.25}
\bbP\left[(1-\delta) s\calS\subseteq\Gamma_{1/2}(t_0+s,\cdot;t_0,z)-z\subseteq (1+\delta) s\calS\,\,\forall s\in [2^{-1}\delta t,2\delta^{-1} t]\right]\geq 1-{\sigma}.
\eeq
Fix any $\Lambda\ge 0$ and let $y_1,\ldots,y_{N''}\in B_\Lambda(0)$ be such that $B_\Lambda(0)\subseteq \bigcup_{i=1}^{N''} B_\eps(y_i)$. For  each $t\geq C_{{\sigma}}$ and  $x_0\in B_{\Lambda t}(0)$, there is $i\in\{1,\ldots,N''\}$ such that $|x_0-ty_i|\leq t\eps $.  Then  \eqref{2.31'} yields for all $s\geq M(2t\eps+1)$,
\beq\lb{2.32}
\begin{aligned}
\Gamma_{1/2}(s-M(2t\eps+1),\cdot ; & M(t\eps+1),ty_i)-ty_i\subseteq\Gamma_{1/2}(s,\cdot;0,x_0)-x_0\\
&\subseteq \Gamma_{1/2}(s+M(2t\eps+1),\cdot;-M(t\eps+1),ty_i)-ty_i
\end{aligned}
\eeq
(using \eqref{2.31'} twice with $s':=Mt\eps\ge M$).
Since $\eps\leq \frac{\delta^2}{10 M}$ and  $t\geq \frac 1\eps$, for any $s\geq \delta t$ we have
\beq\lb{2.34''}
M(3t\eps+2)\le 5Mt\eps\le \frac {\delta^2 t} 2 \le \frac {\delta s} 2,
\eeq
and therefore
\beq\lb{2.34}
(1-2\delta) s\calS \subseteq (1-\delta)(s-M(3t\eps+2))\calS.
\eeq
and
\beq\lb{2.34'}
(1+\delta)(s+M(3t\eps+2))\calS\subseteq (1+2\delta) s\calS 
\eeq
Since also $M(2t\eps+1)\le\delta t\le s$, \eqref{2.32}--\eqref{2.34'} imply
\begin{align*}
    &\bbP \big[(1-2\delta) s\calS\not\subseteq\Gamma_{1/2}(s,\cdot;0,x_0)-x_0 \text{ for some }(s,x_0)\in [\delta t, \delta^{-1}t]\times B_{\Lambda t}(0) \big]\\
    &\qquad \leq\sum_{i=1}^{N''} \bbP\big[(1-\delta) (s-M(3t\eps+2)) \calS \not\subseteq \Gamma_{1/2}(s-M(2t\eps+1),\cdot;M(t\eps+1),ty_i)-ty_i\\
       &\qquad\qquad\qquad\quad  \text{ for some }s\in [\delta t, \delta^{-1}t]\big]
\end{align*}
and
\begin{align*}
    &\bbP \big[ \Gamma_{1/2}(s,\cdot;0,x_0)-x_0\not\subseteq (1+2\delta) s\calS  \text{ for some }(s,x_0)\in [\delta t, \delta^{-1}t]\times B_{\Lambda t}(0) \big]\\
    &\qquad \leq  \sum_{i=1}^{N''} \bbP\big[ \Gamma_{1/2}(s+M(2t\eps+1),\cdot;-M(t\eps+1),ty_i)-ty_i  \not\subseteq (1+\delta) (s+M(3t\eps+2)) \calS \\
          &\qquad\qquad\qquad\quad  \text{ for some }s\in [\delta t, \delta^{-1}t]\big].
\end{align*}
Both right-hand sides are $\leq {\sigma}$, due to \eqref{2.25} with $(s,t_0,z)=(s\mp M(3t\eps+2),\pm M(t\eps+1),ty_i)$ (note  that \eqref{2.34''} shows that $s\pm M(3t\eps+2)\in [2^{-1}\delta t,2\delta^{-1}t]$ when $s\in [\delta t, \delta^{-1}t]$). Thus, after taking ${\sigma}\to0$ and then replacing $\delta$ by $\frac\delta 2$, we obtain \eqref{2.17} with $\theta=\frac12$. The general case follows from parabolic Harnack inequality and  hair-trigger effect as in the proof of Theorem \ref{T.2.5}.
\end{proof}






We now note that it suffices to prove Theorem \ref{T.1.0}(ii) with $y_\eps=0$ for all $\eps\in (0,1)$ because we assume space-time stationarity of $H$.  Recall that for now we assume that $H=(A,b,f)$, although this claim also holds in the general case.  

Let us first assume that  $f$ equals $f'$  from \eqref{0.9}, and that $G$ is bounded.  Then assume without loss that $\delta\in(0,1)$ is such that $G\subseteq B_{\delta^{-1}}(0)$.
Note that the ``unscaled" version of \eqref{3.0'} (with $y_\eps=0$) is
\beq\lb{3.12}
\theta\chi_{(\eps^{-1}G)^0_{\eps^{-1}\rho(\eps)}}\le u_\eps(0,\cdot,\omega)\leq \chi_{B_{\eps^{-1}\rho(\eps)}(\eps^{-1}G)}
\eeq
for $u_\eps(\cdot,\cdot,\omega):=u^\eps(\eps^{-1}\cdot,\eps^{-1}\cdot,\omega)$, which solves \eqref{1.0}.  Let us define 
\[
\Gamma_{\theta',\eps}(t,\omega):=\left\{x\in\bbR^d\,|\, u_\eps(t,x,\omega)\geq \theta'\right\}
\]
 for each $\theta'\in(0,1)$.
Therefore \eqref{3.6} will follow if we prove
\beq\lb{3.6'}
\lim_{\eps\to0}\bbP\left[\eps^{-1}G+\eps^{-1}(1-\delta)t\calS \subseteq \Gamma_{\theta',\eps}(\eps^{-1}t,\cdot) \subseteq \eps^{-1}G +\eps^{-1}(1+\delta)t\calS\,\,\forall t\in [\delta,\delta^{-1}]\right]=1.
\eeq

The hair-trigger effect, \eqref{2.12}, \eqref{3.12}, and the comparison principle imply that there is $M'\geq M$ such that with $s_\eps:=M\eps^{-1}\rho(\eps)+M'$ we have $B_{1}(\eps^{-1}G)\subseteq \Gamma_{1/2,\eps}(s_\eps,\omega)$  for all $(\eps,\omega)\in(0,1)\times\Omega$.
Then $u(0,\cdot,\omega;c_\eps, \eps^{-1}x_0) \leq u_\eps(s_\eps,\cdot,\omega)$ for any $x_0\in G$, so the comparison principle yields for all $(t,x_0)\in\bbR^+\times G$,
\beq\lb{3.9}
\Gamma_{\theta'}(\eps^{-1}t,\omega; c_\eps,\eps^{-1}x_0)\subseteq \Gamma_{\theta',\eps}(\eps^{-1}t+s_\eps,\omega).
\eeq
Since space-time stationarity of $H$ and Theorem \ref{T.2.6} with $(\frac \delta 2, \theta',\frac 1\delta)$ in place of $(\delta,\theta,\Lambda)$ yield
\[
\lim_{\eps\to0}\bbP\left[\eps^{-1}G+\eps^{-1}(1- 2^{-1}\delta) t\calS\subseteq \bigcup_{x_0\in G}\Gamma_{\theta'}(\eps^{-1}t,\cdot;c_\eps,\eps^{-1}x_0)\,\,\forall t\in \big[2^{-1}\delta,2\delta^{-1}\big]\right]=1,
\]
\eqref{3.9} and $\lim_{\eps\to0}\eps c_\eps=0$  show that the probability of just the first inclusion in \eqref{3.6'} converges to $0$ as $\eps\to 0$.

Next, from Theorem 1.2 in \cite{ZlaKPPspread}, and Remarks 2 and 3 following it, we see that 
for each $\delta>0$, 
there is  $C_{\delta}>0$ such that for each $(\eps,\omega)\in(0,1)\times \Omega$ and $t\geq C_\delta $ we get from  \eqref{3.12} that 
\[
u_\eps(t,\cdot,\omega) \le \delta+ \sup_{
z\in B_{\eps^{-1}\rho(\eps)}(\eps^{-1}G)
}u_z \left((1+ 4^{-1}\delta)t,\cdot,\omega\right),
\]
where $u_z$ is a solution to \eqref{1.0} with 
initial data $u_\eps(0,\cdot,\omega)\chi_{B_1(z)}$ (unit cubes were used in \cite{ZlaKPPspread} instead of unit balls, but simple scaling shows that these can be replaced by cubes of side-length $\frac 1d$, which are contained in the corresponding unit balls; recall also that we now have $f=f'$). 
This shows that for any $(\theta',\omega)\in(0,1)\times\Omega$ and $t\geq C_\delta $ we have
\beq\lb{3.10}
\Gamma_{\theta',\eps}(t,\omega)\subseteq \bigcup_{z\in B_{\eps^{-1}\rho(\eps)}(\eps^{-1}G)
}\Gamma_{\theta'-\delta}^z((1+4^{-1}\delta)t,\omega),
\eeq
where $\Gamma_{\theta'}^z(t,\omega):=\{x\in\bbR^d\,|\, u_z(t,x,\omega)\geq \theta'\}$.
Since $\frac12\min\{u,1-u\}\leq \min\{\frac12 u,1-\frac12 u\}$, we see that $\frac12 u_z$ is a subsolution to \eqref{1.0} 
and $\frac12u_z(0,\cdot,\omega)\leq \frac12\chi_{B_1(z)}$.  The comparison principle then shows that
\beq\lb{3.11}
\Gamma_{\theta'-\delta}^z \left((1+4^{-1}\delta)t,\omega \right)\subseteq 
\Gamma_{(\theta'-\delta)/2} \left((1+4^{-1}\delta)t,\omega;0,z \right).
\eeq
Hence \eqref{3.10}, \eqref{3.11}, and Theorem \ref{T.2.6} with 
$(\frac\delta 2, \frac{\theta'-\delta}2, \frac 2\delta)$ in place of $(\delta,\theta,\Lambda)$ yield for any $\delta<\theta'$,
\[
\lim_{\eps\to0}\bbP\left[ 
\Gamma_{\theta',\eps}(t,\cdot)
\subseteq \eps^{-1}B_{\rho(\eps)}(G)+\eps^{-1}(1+2^{-1}\delta)(1+4^{-1}\delta)t\calS\,\,\forall t\in [\delta,\delta^{-1}] \right]=1.
\]
But this, $(1+2^{-1}\delta)(1+4^{-1}\delta)<1+\delta$, and $\lim_{\eps\to0} \rho(\eps)=0$  show that the probability of just the second inclusion in \eqref{3.6'} converges to $0$ as $\eps\to 0$.
Therefore we proved \eqref{3.6'} and hence Theorem \ref{T.1.0}(ii) for all bounded $G$ when $f=f'$.

This now extends to general $f$ because Theorem 1.2 in \cite{ZlaKPPspread} shows that for any $\delta>0$ and all large enough $t$ we have
\[
\pm \big[ u(t,\cdot,\omega;0,0) - u' \left((1\pm \delta)t,\cdot,\omega;0,0\right) \big] \leq \delta,
\]
where $u'(\cdot,\cdot,\omega;0,z)$ solves \eqref{1.0} with $f'$ in place of $f$ and $u'(0,\cdot,\omega;0,z):=\frac12\chi_{B_1(z)}$.  This also shows that $f$ and $f'$ have the same $\calS$, which thus only depends on $H:=(A,b,f_u(\cdot,\cdot,0,\cdot))$.

Finally, the extension to unbounded $G$ is obtained as in the proof of Theorem 1.4 in \cite{ZlaKPPspread} (this uses that, similarly to \eqref{2.12}, perturbations to initial data propagate with speeds $\le M$).  Hence the proof of Theorem \ref{T.1.0}(ii) is finished.

\smallskip
{\it Remark.}  To prove Remark 6 after Theorem \ref{T.1.0}, one can use \eqref{11.2}.  This yields Lemma \ref{L.2.4} with \eqref{2.4} replaced by
\[
\liminf_{r\to\infty}\frac{\tau^{t_0}({x_0,x_0+re},\omega)}{r} \ge \lim_{r\to\infty}\frac{\bbE\left[ \tau^{t_0}({x_0,x_0+re},\cdot)\right]}{r}=\bar{\tau}(e),
\]
which then implies the second claim in \eqref{11.3} as in the proof of Theorem \ref{T.2.5}.  The argument at the end of Section \ref{s.3} then proves the remark.

\section{Proof of Theorem \ref{T.4.0}}\lb{s.5}

Let us fix any $\omega\in\Omega$. Theorem 7.2 in \cite{control} shows that under our hypotheses on $(c,v)$,
\beq\lb{5.15}
u(t,x,\omega):=\sup_{x\in \Gamma(t,\omega;0,y)} u(0,y,\omega)
\eeq
is a viscosity solution to \eqref{G1.1}. We  claim that $u(\cdot,\cdot,\omega)$ is uniformly continuous on $[0,T]\times\bbR^d$ for each $T>0$ whenever $u(0,\cdot,\omega)$ is uniformly continuous. Indeed, let $(t,x,y)\in \bbR^{2d+1}$ be such that $(t,x)$ is $\omega$-reachable from $(0,y)$.
Then there is $\alpha:[0,t]\to \overline{B_1(0)}\subseteq \bbR^d$ and an absolutely continuous path $\gamma:[0,t]\to\bbR^d$ such that $\gamma(0)=y$, $\gamma(t)=x$, and 
\[
\gamma'(s)=v(s,\gamma(s),\omega)+c(s,\gamma(s),\omega) \alpha(s)
\]
for a.e. $s\in[0,t]$. 
Pick any $(\tau,z)\in\bbR^{d+1}$, and extend $\alpha$ to $(t,\tau]$ by $0$ if $\tau>t$.
Then define $\beta:[0,\tau]\to\bbR^d$ via the terminal value problem $\beta(\tau)=z$ and
\[
\beta'(s)=v(s,\beta(s),\omega)+c(s,\beta(s),\omega)\alpha(s)
\]
for a.e. $s\in [0,\tau]$.
If we let $\tau':=\min\{t,\tau\}$ and $y':=\beta(0)$.
The two ODEs yield  
\beq\lb{5.16}
|y-y'|\leq e^{C\tau'}(|x-z|+C|t-\tau|),
\eeq
where 
\[
C:=\|v\|_{L^\infty} + \|\nabla_xv\|_{L^\infty}+\|\nabla_x c\|_{L^\infty}.
\]
It is clear that $(\tau,z)$ is $\omega$-reachable from $(0,y')$, so uniform continuity of $u(\cdot,\cdot,\omega)$ on $[0,T]\times\bbR^d$ for each $T>0$ follows from uniform continuity of $u(0,\cdot,\omega)$.
Then by, e.g. Exercise 3.9 in \cite{bardi}, we obtain that there is a unique uniformly continuous viscosity solution to \eqref{G1.1} for  any  uniformly continuous initial data.

Let us now prove (i). Since $\nabla_x\cdot v(t,\cdot,\omega)=0$ a.e.~for all $(t,\omega)\in\bbR\times\Omega$ and we have \eqref{Gc4}, it follows from Corollary 1.3 in \cite{BIN} that there is $M\geq 1$ such that for any $(t_0,x_0,\omega)\in\bbR^{d+1}\times\Omega$,
\beq\lb{2.12'}
B_{M^{-1}t}(x_0)\subseteq \Gamma(t,\omega;t_0,x_0) \qquad \text{ for all $t\geq M$},
\eeq
with $\Gamma$ from \eqref{G1.3}.
(We note that that result yields \eqref{2.12'} if we replace $c(t,x,\omega)$ in \eqref{G1.1} by the constant $\inf_{(t,x,\omega)\in\bbR^{d+1}\times\Omega}c(t,x,\omega)$, but then \eqref{2.12'} follows from the comparison principle.)
Since $v$ and $c$ are  bounded, after possibly increasing $M$ we also obtain 
\beq\lb{2.12''}
 \Gamma(t,\omega;t_0,x_0)\subseteq B_{{M}t}(x_0)\qquad \text{ for all $t\ge 0$}.
\eeq
Thus \eqref{2.12} holds for all $(t_0,x_0,\omega)\in \bbR^{d+1}\times\Omega$
with $\Gamma$ in place of $\Gamma_{1/2}$. 
We also define the arrival time to a point $x\in\bbR^d$, when starting at $(t_0,x_0)\in\bbR^{d+1}$, by
\beq\lb{G1.4}
\tau^{t_0}(x_0,x,\omega):=\inf\{t\geq 0\,|\,x\in\Gamma(t,\omega;t_0,x_0)\}.
\eeq

The same arguments as in Section \ref{s.3} yield
\eqref{2.31} and \eqref{2.1}--\eqref{2.2}, with $\Gamma$ in place of $\Gamma_{1/2}$. Moreover, \eqref{G1.3} shows that 
if $x'\in \Gamma(t,\omega;t_0,x_0)$, then $\Gamma(s',\omega;t_0+t,x')\subseteq \Gamma(t+s',\omega;t_0,x_0)$ for any $s'\ge 0$.
This with $t+s$ in place of $t$, together with \eqref{2.31} and \eqref{2.12'} with $s'$ in place of $t$, now yields \eqref{2.31'} with $\Gamma$ in place of $\Gamma_{1/2}$.

All this shows that if $H:=(c,v)$ is space-time stationary and 
$\lim_{s\to\infty}s^\alpha \phi_H(s)=0$ for some $\alpha>0$,
then  Lemma \ref{L.2.4} holds with $\tau^{t_0}$ from \eqref{G1.4} and with the same proof. So we can again define $\bar{\tau}$,
$w$, and $\calS$ via \eqref{2.4}, \eqref{2.3'}, and \eqref{2.3}.
The proof of Theorem \ref{T.2.5} also extends to this setting, with $\Gamma$ in place of $\Gamma_{1/2}$ (and without the sentence after \eqref{2.10'}).  We thus obtain for any $\delta\in (0,1)$ and $\Lambda\geq 1$ some $\Omega_{\Lambda,\delta}\subseteq\Omega$ with $\bbP[\Omega_{\Lambda,\delta}]=1$ such that
\beq\lb{2.10''}
x+(1-\delta) t\calS\subseteq\Gamma(t,\omega;0,x) 
\subseteq x+(1+\delta) t\calS
\eeq
holds for each $(x,\omega)\in B_{\Lambda t}(0)\times \Omega_{\Lambda,\delta}$ whenever $t$ is sufficiently large, depending on $\omega,\delta,\Lambda$ 
(this is just \eqref{2.10'} above).
Note that we need neither the parabolic Harnack inequality nor the hair-trigger effect here due to \eqref{G1.3}. 

Fix any $R>0$.  From \eqref{5.5} we obtain
\beq\lb{5.10}
u^\eps(t,x+y_\eps,\omega)=\sup_{x+y_\eps\in \eps\Gamma(\eps^{-1}t,\omega;0,\eps^{-1}y)} u^\eps(0,y,\omega).
\eeq
This, \eqref{5.1'}, \eqref{5.6}, \eqref{2.10''} with $\Lambda+R+M$ in place of $\Lambda$, and \eqref{2.12''} yield that for each $\omega\in\Omega':=\bigcap_{n\in\bbN}\Omega_{n,1/n}$ (so  $\bbP[\Omega']=1$) and $x\in B_R(0)$ we have
\beq\lb{5.9}
\bar{u}((1-\delta)t,x)-\rho(\eps)\leq u^\eps(t,x+y_\eps,\omega)\leq \bar{u}((1+\delta)t,x)+\rho(\eps)
\eeq
for any $\delta>0$ and $t\in[\frac 1R,R]$ whenever $\eps$ is small enough (depending on $\delta,R,\Lambda,\omega$).
If $\varphi$ is a modulus of continuity for $u_0$  (with $\lim_{r\to 0}\varphi(r)=0$), from \eqref{5.6}, \eqref{2.12''}, and $\calS\subseteq B_M(0)$, we also have 
\beq\lb{5.11}
|\bar{u}(t,x)-\bar{u}(t',x')|\leq \varphi(|x-x'|+M|t-t'|)
\eeq
for any $(t,t',x,x',\omega)\in [0,\infty)^2\times\bbR^{2d}\times\Omega$. 
This and \eqref{5.9} show that $u^\eps(\cdot,\cdot+y_\eps,\omega)\to \bar{u}$ locally uniformly on $\bbR^+\times\bbR^d$ as $\eps\to 0$ (for each $\omega\in\Omega'$).
This then easily extends to locally uniform convergence on $[0,\infty) \times\bbR^d$ (i.e., up to time 0), using \eqref{2.12''} with $t_0=0$, $\calS\subseteq B_M(0)$, \eqref{5.5}, \eqref{5.6}, \eqref{5.1'}, and $\lim_{r\to 0}\varphi(r)=\lim_{\eps\to0}\rho(\eps)=0$.
This finishes the proof of (i).


Let us now turn to (ii). As in the proof of Theorem \ref{T.1.0}(ii), it suffices to consider $y_\eps=0$.  Since we have \eqref{2.12}--\eqref{2.2} with $\Gamma$ in place of $\Gamma_{1/2}$,
the proofs of Lemma \ref{L.3.4} and Theorem \ref{T.2.6} with $\Gamma$ in place of $\Gamma_{\theta}$ extend to the present setting (again with no need for the parabolic Harnack inequality or the hair-trigger effect). 
Hence, for any $R>0$ and $\delta\in (0,1)$ we have
\beq\lb{5.8}
\lim_{t\to\infty}\bbP\left[(1-\delta) s\calS\subseteq\Gamma(s,\cdot;0,x_0)-x_0\subseteq (1+\delta) s\calS\,\,\forall (s,x_0)\in [R^{-1} t,R t]\times B_{R t}\right]=1.
\eeq
But then the argument proving \eqref{5.9} again applies, so after using \eqref{5.8} and \eqref{5.11} we obtain
%
\[
\begin{aligned}
\lim_{\eps\to0}\bbP\big[ |u^\eps(t,x,\omega)-\bar{u}(t,x)|\leq \varphi(M\delta R)+\rho(\eps)\,\,\forall (t,x)\in [R^{-1},R]\times B_{R}\big]=1.
\end{aligned}
\]
The result now follows by taking $\delta:=R^{-2}$  and $\eps\to 0$, although with $(t,x)\in [\delta,\delta^{-1}]\times B_{\delta^{-1}}$ inside the probability.  The extension up to time 0 is the same as in part (i).




\end{document}